%% file: volumes.tex
\documentclass[a4paper]{amsart}

\usepackage{amssymb,pstricks,amscd,epsfig}
\usepackage[utf8]{inputenc}
\usepackage[totalwidth=17cm,totalheight=24cm]{geometry}
\usepackage{graphicx}
\input{macros.tex}

\begin{document}

\newcommand{\bsut}{Bers Simultaneous Uniformization Theorem}
\newcommand{\re}{\mathrm{Re}}
\newcommand{\Jd}{\dot{J}}

\newcommand{\dwp}{d_{WP}}

\title{Volumes of quasifuchsian manifolds}

\author{Jean-Marc Schlenker}

\thanks{Partially supported by UL IRP grant NeoGeo and FNR grants INTER/ANR/15/11211745 and OPEN/16/11405402. The author also acknowledge support from U.S. National Science Foundation grants DMS-1107452, 1107263, 1107367 ``RNMS: GEometric structures And Representation varieties'' (the GEAR Network).}
\address{Department of mathematics, 
University of Luxembourg, 
Maison du nombre, 6 avenue de la Fonte,
L-4364 Esch-sur-Alzette, Luxembourg
}
\email{jean-marc.schlenker@uni.lu}

\date{v1, \today}

\begin{abstract}
Quasifuchsian hyperbolic manifolds, or more generally convex co-compact hyperbolic manifolds, have infinite volume, but they have a well-defined ``renormalized'' volume. We outline some relations between this renormalized volume and the volume, or more precisely the ``dual volume'', of the convex core. On one hand, there are striking similarities between them, for instance in their variational formulas. On the other, object related to them tend to be within bounded distance. Those analogies and proximities lead to several questions. Both the renormalized volume and the dual volume can be used for instance to bound the volume of the convex core in terms of the Weil-Petersson distance between the conformal metrics at infinity.
\end{abstract}


\maketitle

\tableofcontents

\section{Two relations between surfaces and quasifuchsian manifolds}
\label{sc:intro}

\subsection{The Teichm\"uller and Fricke-Klein spaces of a surface}

Consider a closed oriented surface $S$ of genus at least $2$, and let $M=S\times \R$. One can then define $\cT_S$, the Teichm\"uller space of $S$, as the space of complex structures on $S$ considered up to diffeomorphisms isotopic to the identity. It is also interesting to introduce $\cF_S$, the {\em Fricke space} of $S$, defined as the space of hyperbolic structures on $S$ considered up to diffeomorphisms isotopic to the identity.

The Poincar\'e-Riemann uniformization theorem provides a diffeomorphism $\cP_S$ between $\cT_S$ and $\cF_S$, but keeping different notations might be preferable here. The geometry of those spaces develops along related but distinct lines. For instance, the cotangent space $T^*_c\cT_S$ at a complex structure $c\in \cT_S$ is classically identifed with the space of holomorphic quadratic differentials on $(S,c)$ and the tangent bundle $T_c\cT_S$ with the space of harmonic Beltrami differentials on $(S,c)$ (see e.g. \cite{ahlfors}), while the cotangent space $T^*_h\cF_S$ at a hyperbolic metric $h$ can be identified with the space of measured geodesic laminations on $(S,h)$, and the tangent space $T_h\cF_S$ is identified with the space of traceless Codazzi 2-tensors on $(S,h)$, see Section \ref{ssc:fischer-tromba}.

The Teichm\"uller space $\cT_S$ of $S$ can be equiped with a Riemannian metric, the Weil-Petersson metric $g_{WP}$. It is simpler to define it on the cotangent space. Given two holomorphic quadratic differentials $q,q'\in T^*_c\cT_S$ at a complex structure $c$, their scalar product is defined as:
$$ g_{WP}(q,q') = \int_S \frac{q\overline{q'}}{h}~, $$
where $h=\cP_S(c)$ is the hyperbolic metric uniformizing $c$.
Here the quotient $q\overline{q'}/h$ makes sense as an area form on $S$, as can be seen using a local coordinate $z$: if $q=fdz^2$ and $q'=f' dz^2$, and if $h=\rho |dz|^2$, then $q\overline{q'}/h=(f\overline{f'}/\rho)|dz|^2$. This scalar product on cotangent vectors defines an identification between the cotangent space $T^*_c\cT_S$ and the tangent space $T_c\cT_S$, a scalar product on the tangent space $T_c\cT_S$, and therefore a Riemannian metric on $\cT_S$. The Weil-Petersson Riemannian metric is know to be K\"ahler \cite{weil:modules,ahlfors61} and to have negative sectional curvature \cite{ahlfors:curvature}. It is incomplete, but geodesically convex \cite{wolpert:noncompleteness,wolpert2}.

On the side of the Fricke space, a closely related analog of the Weil-Petersson metric was defined by Fischer and Tromba \cite{fischer-tromba:wp}. Let $h\in \cF_S$ be a hyperbolic metric, and let $[k_1], [k_2]\in T_h\cF_S$ be the tangent vector fields defined by two traceless, Codazzi symmetric 2-tensors on $(S,h)$. The Fischer-Tromba metric is defined as:
$$ g_{FT}([k_1],[k_2]) = \frac 18\int_S \langle k_1, k_2\rangle_h da_h~. $$
This metric then corresponds to the Weil-Petersson metric, see \cite{fischer-tromba:wp}:
$$ g_{WP} = \cP_S^*g_{FS}~. $$

\subsection{3-dimensional hyperbolic structures}

There are deep relations between the geometry of $\cT_S$ (resp. $\cF_S$) and hyperbolic structures on 3-dimensional manifolds. Those relations develop differently for $\cT_S$ and for $\cF_S$ and remain partly conjectural. Our main goal here is to describe the analogies between those relations. We focus here on quasifuchsian manifolds, and will only briefly mention the extension to convex co-compact hyperbolic manifolds. To simplify notations, we set $M=S\times \R$. Then the boundary $\partial M$ of $(M g)$ can be identified canonically to $S\cup \bar S$, where $\bar S$ is $S$ with the opposite orientation, so that $\cT_{\partial M}=\cT_S\times\cT_{\bar S}$.  Both $\cT_S$ and $\cT_{\bar S}$ can be identified with the space of conformal metrics on $S$, and we will often identify $\cT_S$ with $\cT_{\bar S}$ in this manner.

\begin{defi} \label{df:quasifuchsian}
  A {\em quasifuchsian} structure on $M$ is a complete hyperbolic metric $g$ on $M$ such that $(M, g)$ contains a non-empty compact geodesically convex subset. We denote by $\cQF_S$ the space of quasifuchsian hyperbolic structures on $M$, considered up to isotopies. 
\end{defi}

The relation between quasifuchsian hyperbolic manifolds and the Teichm\"uller space of $S$ rests on the Bers Double Uniformization Theorem, see \cite{bers}.

\begin{theorem}[Bers]
  Given a quasifuchsian structure $g\in\cQF_S$, the asymptotic boundary $\partial_\infty M$ of $(M,g)$ is equipped with a complex structure $c=(c_+, c_-)$, and each such $c\in \cT_{\partial M}$ is obtained from a unique $g\in \cQF_S$.
\end{theorem}

We are also interested in the relation between quasifuchsian manifolds and the Fricke space $\cF_S$. This relation can be understood through a {\em conjectural} statement, due to Thurston, which is analogous to the Bers Double Uniformization Theorem.

\subsection{The convex core of quasifuchsian manifolds}

By definition, a quasifuchsian hyperbolic manifold contains a non-empty, compact, geodesically convex subset. Since the intersection of two non-empty geodesically convex subsets is geodesically convex, any quasifuchsian manifold $(M,g)$ contains a unique smallest non-empty geodesically convex subset, which is compact. It is called the {\em convex core} of $(M,g)$, and will be denoted here by $C(M)$.

There is a rather special case where $C(M)$ is a totally geodesic surface in $(M,g)$ --- in that case, $(M,g)$ is a {\em Fuchsian} manifold. In all other cases, $C(M)$ has non-empty interior, and its boundary is the disjoint union of two surfaces homeomorphic to $S$, denoted here by $\partial_+C(M)$ and $\partial_-C(M)$.

Thurston \cite{thurston-notes} noted that since $C(M)$ is a minimal convex set, its boundary has no extreme point, so $\partial_+C(M)$ and $\partial_-C(M)$ are convex {\em pleated} surfaces. Their induced metrics are hyperbolic (i.e. of constant curvature $-1$), and this defines two points $m_+, m_-\in \cF_S$.

\begin{conj}[Thurston] \label{cj:thurston}
  For all $(m_+, m_-)\in \cF_S\times \cF_S$, there exists a unique $g\in \cQF_S$ such that the induced metrics on $\partial_+C(M)$ and $\partial_-C(M)$ are $m_+$ and $m_-$, respectively.
\end{conj}

The existence part of this statement is known since work of Labourie \cite{L4}, Epstein and Marden \cite{epstein-marden} and Sullivan \cite{sullivan:bourbaki}. 

Conjecture \ref{cj:thurston} is of course analogous to the Bers Simultaneous Uniformization Theorem, when one replaces the complex structure (or conformal metric) at infinity by the induced metric on the boundary of the convex core. One main goal here is to extend this analogy. The other goal is to extend the {\em comparisons} between objects associated to the Teichm\"uller theory of $S$, read at infinity, and objects associated to the Fricke space of $S$, read from the boundary of the convex core. For the conformal metric at infinity and induced metric on the boundary of the convex core, the following result provides a bound on the distance between the two.

\begin{theorem}[Sullivan, Epstein-Marden]
  There exists a universal constant $K$ such that $m_\pm$ are $K$-quasiconformal to $c_\pm$, respectively.
\end{theorem}

The constant $K$ was long conjectured to be equal to $2$, but is actually larger than $2.1$, see \cite{epstein-markovic}.

\subsection{The measured bending lamination of the boundary of the convex core}


To understand the definition of ``dual volume'' that plays a central role below, we need another important notion: the bending measured lamination on the boundary of the convex core. This is the quantity that records in what manner the boundary of the convex core is ``pleated'' in $M$. It is a transverse measure on a geodesic lamination on $\partial C(M)$. A short description of some of its main properties, and of the main properties of measured laminations more generally, can be found in Section \ref{ssc:lamination}.

\subsection{Volumes of quasifuchsian manifolds}

Quasifuchsian hyperbolic manifolds have infinite volume. However, techniques originating from physics make it possible to define a {\em renormalized} volume, see Section \ref{sc:renormalized}. This renormalized volume is closely related to the Liouville functional, see e.g. \cite{takhtajan-zograf:spheres,TZ-schottky,takhtajan-teo,Krasnov:2000zq}. It determines a function $V_R:\cQF _S\to \R$ which can also be considered, through the \bsut, as a function $V_R:\cT_S\times \cT_{\bar S}\to \R$. When $c_-\in \cT_{\bar S}$ is fixed, the function $V_R(\cdot, c_-):\cT_S\to \R$ is a K\"ahler potential for the Weil-Petersson metric on $\cT_S$, a fact that we will not develop here. (A proof can be found in \cite[Section 9]{volume}.)

The convex core $C(M)$, on the other hand, has a well-defined volume, and this defines a function $V_C:\cQF_S\to  \R_{>0}$, the volume of the convex core. It should be clear from the considerations that follow, however, that $V_C$ is not the ``right'' analog of the renormalized volume, and we rather consider the {\em dual volume}.

\begin{defi} \label{df:V*}
  The {\em dual volume} of the convex core of a quasifuchsian manifold $(M,g)$ is
  $$ V^*_C(M) = V_C(M) -\frac 12 L_m(l)~, $$
  where $m$ is the induced metric on the boundary of the convex core, and $l$ is it's measured bending lamination.
\end{defi}

The dual volume can be defined for a more general geodesically convex subset $K\subset M$. For a convex subset with smooth boundary, it is defined as
$$ V^*(K)=V(K) -\frac 12\int_{\partial K} Hda~, $$
where $H$ is the mean curvature of $\partial K$ (defined as the sum of its principal curvatures) and $da$ is the area form of the induced metric on the boundary of $K$.

The reason for the term ``dual volume'' is that, if $P$ is a convex polyhedron in $\HH^3$ and its ``dual volume'' $V^*$ is defined in the same manner as $V^*=V-\sum_e L_e\theta_e$, where the sum is over the edges and $L_e$ (resp. $\theta_e$) is the length (resp. exterior dihedral angle) of edge $e$, then $V^*$ is equal to the volume, suitably defined, of the dual polyhedron in the de Sitter space, see \cite{HR}. For quasifuchsian manifolds, a similar interpretation is possible, but only in a {\em relative} manner. A quasifuchsian manifold has a de Sitter counterpart $M^*$, which is a pair of globally hyperbolic de Sitter manifolds $M^*_+, M^*_-$, see \cite{mess,mess-notes}. Any convex compact subset $K\subset M$ has a pair of dual convex subsets $K^*_+\subset M^*_+, K^*_-\subset M^*_-$.  If $K$ and $\bar K$ are two subsets of $M$ with $K\subset \bar K$, then $\bar K^*\subset K^*$, and $V^*(\bar K)-V^*(K)=V(K^*\setminus\bar K^*)$. We do not delve more onto this topic and refer the intersted reader to \cite{mazzoli2019} for details.


\subsection{The dual Bonahon-Schl\"afli formula}

The classical Schl\"afli formula \cite{milnor-schlafli} expresses the first-order variation of the volume of a hyperbolic polyhedron $P\subset \HH^3$ in terms of the variation of its exterior dihedral angles as follows:
$$ \dot P = \frac 12 \sum_e l(e)\dot \theta(e)~, $$
where the sum is over the edges of $P$, $l(e)$ is the length and $\theta(e)$ the exterior dihedral angle of edge $e$.

Bonahon \cite{bonahon-variations,bonahon} extended this classical formula to the convex cores of quasifuchsian (or more generally convex co-compact) hyperbolic manifolds. In a first-order deformation of a quasifuchsian manifold $(M,g)$, corresponding say to a first-order variation of the holonomy representation, 
\begin{equation}
  \label{eq:BS}
  \dot V = \frac 12 L_m(\dot l)~.
\end{equation}
Bonahon showed that $\dot l$, the first-order variation of $l$, makes senses as a {\em H\"older cocycle}, and has a well-defined length, so that \eqref{eq:BS} makes sense.

The dual Bonahon-Schl\"afli formula is the analog of the Bonahon-Schl\"afli variational formula for the dual volume (see \cite{cp}). It is a direct consequence of \eqref{eq:BS}:
\begin{equation}
  \label{eq:BS*}
  \dot V^*_C = -\frac 12 (dL(l))(\dot m)~. 
\end{equation}

Note however that the interpretation of \eqref{eq:BS*} is much simpler than that of \eqref{eq:BS}, since now the right-hand term is simply the differential of an analytic function --- the length of $l$ --- applied to a tangent vector to $\cF_{\partial C(M)}$. We will see below that Equation \eqref{eq:BS*} is closely analogous to the variational formula for the renormalized volume. Before stating this formula, we need to better understand the geometric data at infinity of quasifuchsian manifolds.

\subsection{The holomorphic quadratic differential at infinity}

We now introduce what we believe to be a natural analog at infinity of the measured lamination on the boundary of the convex core. This is a measured {\em lamination}, defined as follows. Given a quasifuchsian structure $g\in \cQF_S$ on $M$, we have seen that the asymptotic boundary $\partial_\infty M$ is the disjoint union of two disjoint Riemann surfaces $(S, c_+)$ and $(\bar S, c_-)$. In fact, each of those surfaces is equiped not only with a complex structure $c_\pm$, but also with a {\em complex projective structure} $\sigma_\pm$, see Section \ref{ssc:quasifuchsian}.

The Schwarzian derivative (see Section \ref{ssc:schwarzian}) provides the tool to compare $\sigma_\pm$ to $\sigma_F(c_\pm)$, the Fuchsian complex projective structure associated to $c_\pm$. This yields a holomorphic quadratic differential $q_\pm$ on $(S,c_+)$ and $(\bar S,c_-)$, or in other terms a holomorphic quadratic differential $q$ on $\partial_\infty M$, which we call the {\em holomorphic quadratic differential at infinity}.

\subsection{A first variational formula for the renormalized volume}

The renormalized volume also satisfies a simple variational formula, see Section \ref{ssc:variational}.
\begin{equation}
  \label{eq:dVR}
 \dot V_R = \re(\langle q,\dot c\rangle)~, 
\end{equation}
where $q$ is considered as a vector in the complex cotangent to 
$\cT_S$ at $c$, and $\langle \cdot, \cdot\rangle$ is the duality 
bracket.

We will see below that this first variational formula can be formulated in a way that makes it similar to \eqref{eq:BS*}, using the extremal length of a measured foliation at infinity instead of the hyperbolic length of a measured lamination on the boundary of the convex core.

\subsection{The measured foliation at infinity and Schl\"afli formula at infinity}

A holomorphic quadratic differential $q$ on a Riemann surface $(S,c)$ determines canonically two measured foliations, the {\em horizontal} and {\em vertical} foliations. The leaves of the horizontal (resp. vertical) foliation are the integral curves of the vector fields $u$ such that $q(u,u)\in \R_{>0}$ (resp. $\in R_{<0}$), see \cite{FLP}.

\begin{defi}
The {\em measured foliation at infinity} of $M$, denoted by $f\in \cMF_{\partial M}$, is the horizontal foliation of the holomorphic quadratic differential $q$ of $M$. 
\end{defi}

\subsection{The Schl\"afli formula for the renormalized volume}

There is a simple variational formula for the renormalized volume, in terms of $q$ and of the variation of the conformal structure at infinity, Equation \eqref{eq:dVR}. Here we write this variational formula in another way, involving the measured foliation at infinity. Instead of the hyperbolic length of the measured bending lamination, as for the dual volume, this formula involve the {\em extremal length} of the measured foliation at infinity.

Recall that that given a Riemann surface $(S,c)$ and a simple closed curve $\gamma$ on $S$, the extremal length $\ext(\gamma)$ of $\gamma$ can be defined as the supremum of the inverses of the conformal moduli of annuli embedded in $S$ with meridian isotopic to $\gamma$.
 
\begin{theorem} \label{tm:schlafli}
  \label{tm:f}
  In a first-order variation of $M$, we have
  \begin{equation}
    \label{eq:schlafli}
     \dot V_R = - \frac 12 (d\ext(f))(\dot c)~. 
  \end{equation}
\end{theorem}

Here $\ext(f)$ is considered as a function over the Teichm\"uller space of the boundary $\cT_{\partial M}$. The right-hand side is the differential of this function, evaluated on the first-order variation of the complex structure on the boundary.

\subsection{Comparing and relating the two viewpoints}

Theorem \ref{tm:schlafli}, and the analogy between \eqref{eq:BS*} and \eqref{eq:schlafli}, suggests an analogy between the properties of quasifuchsian manifolds considered from the boundary of the convex core and from the boundary at infinity. For instance, on the  boundary of the convex core, we have the following upper bound on the length of the bending lamination, see \cite[Theorem 2.16]{bridgeman-brock-bromberg}.

\begin{theorem}[Bridgeman, Brock, Bromberg] \label{tm:6pi}
$L_{m_\pm}(l_\pm)\leq 6\pi|\chi(S)|$.
\end{theorem}

Similarly, on the boundary at infinity, we have the following result, proved in Section \ref{ssc:foliation}.

\begin{theorem} \label{tm:ext}
$\ext_{c_\pm}(f_\pm)\leq 3\pi|\chi(S)|$.
\end{theorem}

\begin{table}[ht]
  \centering
  \begin{tabular}{|c|c|}
    \hline
    On the convex core & At infinity \\
    \hline\hline
    Induced metric $m$ & Conformal structure at infinity $c$ \\
    \hline
    Thurston's conjecture on prescribing $m$ & Bers' Simultaneous Uniformization Theorem \\
    \hline 
    Measured bending lamination $l$ & measured foliation $f$ \\
    \hline
    Hyperbolic length of $l$ for $m$ & Extremal length of $f$ for $c$ \\
    \hline
    Volume of the convex core $V_C$ & Renormalized volume $V_R$ \\
    \hline
    Dual Bonahon-Schl\"afli formula & Theorem \ref{tm:schlafli} \\
    $\dot V^*_C = -\frac 12 (dL(l))(\dot m)$ & $\dot V_R = -\frac 12 (d\ext(f))(\dot c)$ \\
    \hline
    Bound on $L_m(l)$ \cite{bridgeman,bridgeman-brock-bromberg} & Theorem \ref{tm:ext} \\
    $L_{m_\pm}(l_\pm)\leq 6\pi|\chi(S)|$ & $\ext_{c_\pm}(f_\pm)\leq 3\pi|\chi(S)|$ \\
    \hline
    Brock's upper bound on $V_C$ \cite{brock:2003} & Upper bound on $V_R$ \cite{compare_v4} \\
    \hline
  \end{tabular}
  \caption{Infinity vs the boundary of the convex core}
  \label{tb:1}
\end{table}


This analogy, briefly described in Table \ref{tb:1}, suggests a number of questions (see Section  \ref{sc:questions}) since it appears that, at least up to a point, results known on the boundary of the convex core might hold also on the boundary at infinity, and conversely.

Another series of questions stems from comparing the data on the boundary of the convex core to the corresponding data on the boundary at infinity. For instance, it was proved by Sullivan that the induced metric on the boundary of the convex core is uniformly quasi-conformal to the conformal metric at infinity (see \cite{epstein-marden,epstein-markovic}), and one can ask whether similar statements hold for other quantities. We do not delve much  on those questions here, see Section \ref{ssc:extension} for a question in this direction.

\subsection*{Outline of the content}

Section 2 contains background material on a variety of topics that are considered or used in the paper. The renormalized volume is defined in Section 3, and its main properties proved. Section 4 contains the proof of the Schl\"afli-type formula for the renormalized volume, \eqref{eq:schlafli}, while section \ref{sc:applications} explains how to obtain upper bounds on the volume of the convex core in terms of boundary data, using either the dual or the renormalized volume. It then outlines some applications, in particular results of Brock and Bromberg \cite{brock-bromberg:inflexibility2} on the systoles of the Weil-Petersson metric on moduli space and of Kojima and McShane \cite{kojima-mcshane} on the comparison between the entropy of a pseudo-Anosov diffeomorphism and the hyperbolic volume of its mapping torus. Finally Section \ref{sc:questions} presents some open questions. 

\section{Background material}

This section contains a short description of some of the background material used in the paper, aiming at providing references for readers who are not familiar with certain topics.

\subsection{The Fischer-Tromba metric} \label{ssc:fischer-tromba}

Let $h$ be a hyperbolic metric on $S$. The tangent space $T_h\cF_S$ to the Fricke space of $S$ can be identified with the space of symmetric 2-tensors on $S$ that are traceless and satisfy the Codazzi equation for $h$, see \cite{fischer-tromba:wp}. (In other terms, the real parts of holomorphic quadratic differentials in $\cQ_c$, if $c$ is the complex structure compatible with $h$ on $S$.) 

Let $k,l$ be two such tensors and let $[k],[l]$ be the corresponding vectors in  $T_h\cF$. Then the Weil-Petersson metric between $[k]$ and $[l]$ can be expressed as
$$ \langle [k],[l]\rangle_{WP} = \frac 18\int_S \langle k,l\rangle_h da_h~. $$
The right-hand side of this equation is sometimes called the Fischer-Tromba metric on $\cF_S$. It is proved in \cite{fischer-tromba:wp} that this metrics corresponds to the Weil-Petersson metric on $\cT_S$, through the identification of $\cT_S$ with $\cF_S$ by the Poincaré-Riemann Uniformization Theorem. 


We can also relate the scalar product on symmetric 2-tensors to the natural bracket between holomorphic quadratic differentials and Beltrami differentials as follows.

\begin{lemma} \label{lm:coef4}
Let $X$ be a closed Riemann surface, and let $h$ be the hyperbolic metric compatible
with its complex structure. Let $\dot h$ be a first-order deformation of $h$, and 
let $\mu$ be the corresponding Beltrami differential. Then for any holomorphic
quadratic differential $q$ on $X$,
$$ \int_X \langle Re(q), h'\rangle_h da_h = 4Re\left(\int_X q\mu\right)~. $$
\end{lemma}


\subsection{Complex projective structures on a surface}

A complex projective structure (also called $\C P^1$-structure) is a $(G,X)$-structure (see \cite{thurston-notes,goldman:geometric}), where $X=\C P^1$ and $G=PSL(2, \C)$. Such a structure can be defined by an atlas of charts with values in $\C P^1$, with change of coordinates in $PSL(2, \C)$. We denote by $\cCP_S$ the space of $\C P^1$-structures on $S$.

The space $\cCP_S$ of complex projective structures can be identified with either $T^*\cT_S$ or $T^*\cF_S$, itself identified with $\cF_S\times \cML_S$. We describe those two identifications below. The first uses the Schwarzian derivative, while the second is through the grafting map.

\subsection{The Schwarzian derivative} 
\label{ssc:schwarzian}

Let $\Omega\subset \C$ be an open subset, and let $f:\Omega\to \C$ be holomorphic. The Schwarzian
derivative of $f$ is a meromorphic quadratic differential defined as
$$ \cS(f) = \left(\left(\frac{f''}{f'}\right)' -
\frac 12\left(\frac{f''}{f'}\right)^2\right) dz^2~. $$
It has two remarkable properties that both follow from the lengthy but direct computations based on the definition. 
\begin{enumerate}
\item $\cS(f)=0$ if and only if $f$ is a M\"obius transformation,
\item if $g:\Omega'\to \C$ is holomorphic and $f(\Omega)\subset \Omega'$ then $\cS(g\circ f)=f^*\cS(g)+\cS(f)$.
\end{enumerate}
It follows from those two properties that the Schwarzian derivative is defined for any holomorphic map from a surface equiped with a complex projective structure to another: given such a map, its Schwarzian derivative can be computed with respect to a coordinate chart in the domain and target surfaces, and properties (1) and (2) indicate that it actually does not depend on the choice of charts.

There are several nice geometric interpretations of the Schwarzian derivative that can be found in \cite{thurston:zippers}, \cite{epstein:reflections} or in \cite{dumas-duke}.

Given a complex structure $c\in \cT_S$ on $S$, there is by the Poincar\'e-Riemann Uniformization Theorem a unique hyperbolic metric $h_c$ on $S$ compatible with $c$. Any hyperbolic metric has an underlying complex projective structure on $S$, because the hyperbolic plane can be identified with a disk in $\C P^1$, on which hyperbolic isometries act by elements of $PSL(2, \C)$ fixing the boundary circle.
We denote by $\sigma_F(c)$ the underlying complex projective structure of the hyperbolic metric $h_c$ , and call it the {\em Fuchsian} complex projective structure of $c$.

Let $\sigma \in \cCP_S$, and let $c\in \cT_S$ be the underlying complex structure. There is a unique map $\phi:(S,\sigma)\to (S,\sigma_F(c))$ holomorphic for the underlying complex structure and isotopic to the identity. Let $q(\sigma)=\cS(\phi)$ be its Schwarzian derivative. This construction defines a map $\cQ:\cCP_S\to T^*\cT_S$, sending $\sigma$ to $(c,q(\sigma))$.

The holomorphic quadratic differential $q(\sigma)$ can be considered as a cotangent vector to $\cT_S$ at $c$, so that $\cQ$ can be defined as a map from $\cCP_S$ to $T^*\cT_S$.

The map $\cQ$ is known to be a homeomorphism, see \cite{dumas-survey}. 

\subsection{The measured bending lamination on $\partial C(M)$}
\label{ssc:lamination}

Although the induced metric on the boundary of the convex core is hyperbolic, the boundary surface is not (except in the Fuchsian case) totally geodesic. Rather, it is ``pleated'' along a locus which is a disjoint union of complete geodesics.

The simplest situation is when this pleating locus is a simple closed geodesic, or a disjoint union of such geodesics. The amount of pleating is then measured by an angle, analogous to the exterior dihedral angle at the edge of a hyperbolic polyhedron. It is quite natural then to describe the pleating as a {\em transverse measure} along the pleating locus: any segment transverse to the pleating locus has a weight, which is simply the sum of the pleating angles along the connected component of the pleating locus that it intersects, and this weight is constant when the segment is deformed while remaining transverse to the pleating locus. There is then a natural notion of ``length'' of this measured pleating lamination: it is simply the sum of products of the length of the connected component of the pleating locus by their pleating angle.

However, the pleating locus is generally much more complicated: it is a {\em geodesic lamination}, that is, disjoint union of geodesics who might be non-closed. This geodesic lamination is also equiped with a transverse measure quantifying the amount of pleating. The pleating of the surface is therefore described by a {\em measured geodesic lamination}. 

We refer the reader to \cite{bonahon-geodesic} for a nice introduction to geodesic laminations on hyperbolic surfaces. Here are a few key points.
\begin{itemize}
\item As for closed curves, the notion of measured lamination can be considered on a surface without reference to a hyperbolic metric. Given a measured lamination on $S$, it has a unique geodesic realization for each hyperbolic metric $h$ on $S$. We will denote by $\cML_S$ the space of measured laminations on $S$.
  \item $\cML_S$ can be defined as the completion of the space of weighted closed curves (or multicurves) on $S$ for a natural topology defined by intersection with closed curves. 
  \item The projectivization $P\cML_S$ of $\cML_S$ provides a compactification of $\cF_S$, called the Thurston boundary, see e.g. \cite{thurston-notes,FLP}.
  \item On a hyperbolic surface $(S,h)$, measured laminations have a well-defined length, defined by continuity from the length of weighted closed curves. (The length of a weighted closed curve is the product of the weight by the length of the geodesic representative of the curve.) We will denote by $L_h(l)$ the length of a measured lamination $l$ with respect to a hyperbolic metric $h$ on $S$. 
  \item For each $l\in \cML_S$, the length function $L_\cdot(l):\cF_S\to \R_{\geq 0}$ is analytic over $\cF_S$ (see \cite{kerckhoff:analytic}). At each hyperbolic metric $h\in \cF_S$, the derivative $l\to d_hL_\cdot(l)$ provide a homeomorphism from $\cML_S$ to $T^*_h\cF_S$, so that $T^*\cF_S$ can be identified globally with $\cF_S\times \cML_S$.
\end{itemize}

\subsection{The grafting map}

We now turn to the description of complex projective structures on a surface in terms of hyperbolic metrics and measured laminations.


Consider first the simple situation where the measured bending lamination on $\partial_+C(M)$ is supported on a disjoint union of closed curves. The upper boundary at infinity $\partial_+M$ of $M$ can then be decomposed as the union of two sub-domains by considering the extension to infinity of the nearest-point projection from $M\setminus C(M)$ to $\partial C(M)$. The set of point which project to the complement of the bending locus of $\partial_+C(M)$ is projective equivalent to the complement of a lamination in $\partial_+C(M)$ (equiped with the complex projective structure underlying its induced hyperbolic metric), while the set of points projecting to the bending lamination of $\partial_+C(M)$ is a disjoint union of annuli, each carrying a standard complex projective structure depending on two parameters: the length and bending angle at each closed geodesic in the bending locus.

In this manner, the complex projective structure on $\partial_+M$ can be obtained by a well-defined procedure, where $\partial_+C(M)$ (equiped with the complex projective structure underlying its induced metric) is cut along the support of the measured lamination, and a projective annulus is inserted in each cut. Thurston called {\em grafting} the function sending the induced metric $m$ and measured bending lamination $l$ to the complex projective structure at infinity $\sigma$, and he proved that this function extends to a {\em homeomorphism} $gr:\cF_S\times \cML_S\to \cCP_S$, see \cite{dumas-survey,kamishima-tan}.

The grafting map therefore provides an identification of $\cCP_S$ and $T^*\cF_S$, identified with $\cF_S\times \cML_S$. 

\subsection{The energy of harmonic maps and the Gardiner formula}
\label{ssc:energy}

The proof of Equation \eqref{eq:schlafli} from Equation \eqref{eq:dVR} uses well-known results involving the energy of harmonic maps and length of measured foliations. We recall those statements in this section and the next.


Given a measured foliation $f\in \cMF_S$, consider its universal cover $\tilde{f}$, which is a measured foliation of $\tilde{S}$. One can then define the {\em dual tree} $T_{\tilde{f}}$ of the universal cover $\tilde{f}$, see e.g. \cite{wolf:mathz,morgan-shalen:I}. In the simplest case where $f$ has closed leaves, the vertices of the dual tree $T_{\tilde{f}}$ correspond to singular points of the foliation $\tilde{f}$, while each leave of the foliation corresponds to an interior point of an edge. However for general measured foliations, $T_{\tilde{f}}$ is a {\em real tree}.

Let $f\in \cMF_S$ be a measured foliation, and let $T_f$ be its dual real tree. For each $c\in \cT_S$, there is a unique equivariant harmonic map $u$ from $\tilde S$ to $T_f$, see \cite{wolf:mathz}. Let $E_f(c)=E(u,c)$ be its energy, and let $\Phi_f$ be its Hopf differential. The following remarkable formula can be found in \cite[Theorem 1.2]{wentworth:gardiner}.
\begin{equation}
  \label{eq:gardiner}
   dE_f(\dot c) = -4 Re(\langle\Phi_f,\dot c\rangle)~. 
\end{equation}
Here $\dot c$ is considered as a Beltrami differential, and $\langle \cdot, \cdot\rangle$ is the duality product between Beltrami differentials and holomorphic quadratic differentials. 

We use below the same notations, but with $S$ replaced by $\partial M$.

\subsection{Extremal lengths of measured foliations}

Let $c$ be a complex structure on $S$, and let $Q$ be a holomorphic quadratic differential on $(S,c)$. $Q$ determines two measured foliations on $S$, its {\em horizontal} and {\em vertical} foliations, see \cite{hubbard-masur:acta79}. For any non-zero vector $v$ tangent to a leaf of the horizontal foliation, $q(v,v)\in \R_{>0}$, while if $v$ is tangent to a leaf of the vertical foliation, $q(v,v)\in \R_{<0}$. It is well known (see \cite{}) that, given a measured foliation $f$ on $(S,c)$, there is a unique holomorphic quadratic differential on $(S,c)$ with horizontal measured foliation $f$.

Let $f$ be a measured foliation on $S$ and, for given $c\in \cT$, let $Q$ be the holomorphic quadratic differential on $S$ with horizontal foliation $f$. We will use the following relation, see \cite{kerckhoff:asymptotic}.

\begin{lemma} \label{lm:Q}
The {\em extremal length} of $f$ at $c$ is the integral over $S$ of $Q$,
$$ \ext_c(f) = \int_S |Q|~. $$
\end{lemma}

Moreover, Wolf proved that the extremal length of a measured foliation is directly related to the energy of the harmonic map to its dual tree as follows.

\begin{theorem}[\cite{wolf:realizing}] \label{tm:wolf}
$Q=-\Phi_f$. Moreover,
$$ E_f(c) = 2\int_S|\Phi_f| = 2\int_S|Q| = 2 \ext_c(f)~. $$ 
\end{theorem}

\subsection{Quasifuchsian manifolds}
\label{ssc:quasifuchsian}

We collect here a few basic facts on quasifuchsian hyperbolic manifolds. Recall that $M=S\times \R$, where $S$ is a closed oriented surface of genus at least $2$. 

Quasifuchsian structures on $M$ were defined in Definition \ref{df:quasifuchsian}, but can also be defined as quasiconformal deformations of Fuchsian structures. Specifically, given a complete hyperbolic metric $g$ on $M$, with $(M,g)$ isometric to $\HH^3/\rho(\pi_1S)$, $(M,g)$ is quasifuchsian if and only if there exists a Fuchsian representation $\rho_0:\pi_1S\to \PSL(2,\R)$ and a quasiconformal homeomorphism $\phi:\CP^1\to \CP^1$ such that the actions of $\rho_0$ and $\rho$ on $\CP^1$ are conjugated by $\phi$: $\rho(\gamma)=\phi^{-1}\circ \rho_0(\gamma)\circ \phi$ for any $\gamma\in \pi_1S$.

This point of view leads to the following proposition. 

\begin{prop}
  Given a quasifuchsian structure $g\in \cQF_S$ on $M$, $(M,g)$ is the quotient of $\HH^3$ by the image of a morphism $\rho:\pi_1S\to \PSL(2,\C)$. The corresponding action of $\pi_1S$ on $\CP^1$ is properly discontinuous and free on each connected component of the complement of a Jordan curve $\Lambda_\rho$. Moreover, $\Lambda_\rho$ is a {\em quasicircle}, that is, the image of $\RP^1\subset  \CP^1$ by a quasiconformal homeomorphism from $\CP^1$ to $\CP^1$.
\end{prop}

It follows that each connected component of $\CP^1\setminus \Lambda_\rho$ corresponds to a connected component of the boundary at infinty $\partial_\infty M$. Since $\rho$ acts on each by elements of $\PSL(2, \C)$, each is equiped with a complex projective structure. We denote by $\sigma$ the complex projective structure defined in this manner on $\partial_\infty M$, and by $\sigma_{\pm}$ the complex projective structure defined on $\partial_{\infty, \pm}M$, the two connected components of $\partial_\infty M$.




\section{The renormalized volume of quasifuchsian manifolds} \label{sc:renormalized}

\subsection{Outline}

The renormalized volume of quasifuchsian manifolds is at the intersection of two distinct developments in mathematics.
\begin{itemize}
\item It is closely related to the Liouville functional in complex analysis, see \cite{takhtajan-zograf:spheres,TZ-schottky,takhtajan-teo,Krasnov:2000zq}.
\item It can also be considered as the 3-dimensional case of the renormalized volume of conformally compact Einstein manifolds, see \cite{Skenderis,graham-witten,graham}.
\end{itemize}

A definition of the renormalized volume of quasifuchsian manifolds can be found in \cite[Def 8.1]{volume} or in \cite[Section 3]{compare}. We recall this definition here for completeness. It is based on equidistant foliations in the neighborhood of infinity, on a notion of ``$W$-volume'' of geodesically convex subsets of a quasifuchsian manifold, and on a ``renormalized'' limit as $r\to \infty$ of the $W$-volume of the region between the level $r$ surfaces of a well-chosen equidistant foliation.

Another equivalent definition uses a conformal description of the metric and the behavior at $0$ of a meromorphic function constructed by integration, see \cite{guillarmou-moroianu-rochon}.

\subsection{Equidistant foliations near infinity} \label{ssc:equidistant}

We first define equidistant foliations in the neighborhood of infinity in a quasifuchsian manifold.

\begin{defi}
  An {\it equidistant foliation} of $M$ near $\partial_{\infty,+}M$ (resp. $\partial_{\infty,-}M$) is a foliation of a neighborhood of $\partial_{+,\infty} M$ (resp. $\partial_{\infty,-}M$) by locally convex surfaces, $(S_r)_{r\geq r_0}$,  for some $r_0>0$, such that, for all $r'>r\geq r_0$, $S_{r'}$ is between $S_r$ and $\partial_{\infty,+}M$, and at constant distance $r'-r$ from $S_r$.
\end{defi}

Two equidistant foliations in $E$ will be identified if they coincide in a neighborhood of infinity. In this case they can differ only by the  first value $r_0$ at which they are defined.

Given an equidistant foliation $(S_r)_{r\geq r_0}$ and given  $r'>r\geq 0$, there is a natural identification between $S_r$ and $S_{r'}$, obtained by following the normal direction from $S_t$ for all $t\in [r,r']$. This identification will be implicitly used below. 

\begin{defi} \label{df:I*}
Let $(S_r)_{r\geq r_0}$ be an equidistant foliation of $M$ near $\partial_{\infty,+}M$ (resp. $\partial_{\infty,-}M$). The {\it metric at infinity}, {\em second} and {\em third fundamental forms} at infinity associated to $(S_r)_{r\geq r_0}$ are defined by the asymptotic development:
\begin{equation}
  \label{eq:I*}
  I_r = \frac 12 (e^{2r}I^* + 2\II^* + e^{-2r}\III^*)~,
\end{equation}
where $I_r$ is the induced metric on $S_r$. 
\end{defi}

Those symmetric 2-tensors $I^*, \II^*$ and $\III^*$ can naturally be defined as a metric on $\partial_{\infty,+}M$ (resp. $\partial_{\infty,-}M$). The existence of the asymptotic development follows from a straightforward computation using the expansion of $I_r$ as a function of $r$, see \cite{volume}. This direct computation shows that, if $S_0$ exists and is smooth, then
\begin{equation}
  \label{eq:I-I*}
 I^* = \frac 12 (I+2\II+\III)~, ~~ \II^* = \frac 12(I-\III)~, ~~ \III^* = \frac 12 (I-2\II+\III)~,
\end{equation}
where $I, \II$ and $\III$ are the induced metric and second and third fundamental forms of $S_0$.

The first part of the following proposition is quite elementary (see e.g. \cite{volume}) while the second part follows from ideas of Epstein \cite{Eps}, see below.

\begin{prop} \label{pr:33}
The limit metric $I^*$ is in the conformal class at infinity of $M$. 

Let $M$ be a quasifuchsian manifold, and let $h$ be a Riemannian metric on $\partial_{\infty,+}M$ (resp. $\partial_{\infty,-}M$) in the conformal class at infinity of $M$. There is a unique equidistant foliation in near $\partial_{\infty,+}M$ (resp. $\partial_{\infty,-}M$) such that the associated metric at infinity $I^*$ is equal to $h$.
\end{prop}

This equidistant foliation can be defined from a metric at infinity in terms of envelope of a family of horospheres, see \cite{Eps}. We briefly outline this construction here for completeness. Consider the hyperbolic space $\HH^3$ as the universal cover of $M$. The metric $I^*$ lifts to a metric on the domain of discontinuity $\Omega$ of $M$, in the canonical conformal class of $\partial_\infty \HH^3$. Let $x\in \Omega$. For each $y\in \HH^3$, the visual metric $h_y$ on $\partial_\infty \HH^3$ is conformal to $I^*$. Let $H_{x,r}$ be the set of points $y\in \HH^3$ such that $h_y\geq e^{2r}I^*$ at $x$. A simple computation shows that $H_{x,r}$ is a horosphere intersecting $\partial_\infty \HH^3$ at $x$, and the lift of $S_r$ to $\HH^3$ happens to be equal to the boundary of the  union of horoballs bounded by the $H_{x,r}$, for $x\in \Omega$.

An alternative approach is provided in \cite{horo}, in terms of the isometric embedding of the metric $h$ in the ``space of horospheres'' of $\HH^3$, an of a duality between this ``space of horospheres'' and $\HH^3$. 

\subsection{Definition and first variation of the $W$-volume}

Consider a quasifuchsian manifold $M$ and  a geodesically convex subset $N$ of $M$ with smooth boundary. We first define (in Definition \ref{df:W}) a modified volume of $N$, and will then use this modified volume, for a particular choice of a convex subset of $M$, to define the renormalized volume of $M$ (Definition \ref{df:renormvol}).

\begin{defi} \label{df:W}
Let $N\subset M$ be a convex subset. We set:
$$ W(N) = V(N) -\frac{1}{4}\int_{\dr N} H da $$
where $H$ is the mean curvature of $\partial N$ and $da$ is the area form of
its induced metric.
\end{defi}

There is a clear similarity between this $W$-volume and the dual volume of convex subsets of $M$ seen above: only the coefficient changes. The $W$-volume can thus be considered as the half-sum of the volume and dual volume.

The first variation of this modified volume is computed in \cite{volume}, using an earlier variation formula for deformations of Einstein manifolds with boundary \cite{sem,sem-era}. Here we consider a first-order deformation of the hyperbolic metric on $N$, and denote by $I'$ and $\II'$, respectively, the corresponding first-order variations of the induced metric and second fundamental form on the boundary of $N$, and denote the derivatives of all quantities with a prime. Here we consider a first-order variation of the hyperbolic metric on $N$, that is, we do not only vary $N$ as a convex subset of $M$ but also allow variations of $M$. 

\begin{lemma} \label{lm:var1}
Under a first-order deformation of $N$,
\begin{equation} \label{eq:var1}
   W' = \frac{1}{4}\int_{\dr N} \langle \II' - \frac{H}{2} I',I\rangle_I da_I~. 
\end{equation}
\end{lemma}

\begin{proof}
  It was proved in \cite{sem,sem-era} that in this setting the first-order variation of the volume is given by:
  $$ 2V' = \int_{\partial N} H' + \frac 12\langle I',\II\rangle_{I} da_I~. $$
  The first-order variation of the area form of $I$ is equal to
  $$ da_I' = \frac 12 \langle I', I\rangle_I da_I~, $$
  and it follows from the definition of $W(N)$ that
  $$ W' = V' - \frac 14\left(\int_{\partial N} H da_I\right)' = \int_{\partial N} \frac {H'}4 + \frac 14 \langle I', \II\rangle_I - \frac 18 H\langle I',I\rangle_I da_I~. $$
  However a simple computation shows that
  $$ H' = (\langle \II,I\rangle_I)' = \langle \II',I\rangle_I - \langle \II, I'\rangle_I~, $$
  and the result follows.
\end{proof}

The scalar product appearing in \eqref{eq:var1} and in the proof between symmetric bilinear forms is the usual extension to tensors of the Riemannian scalar product on $T\partial N$ defined by the induced metric $I$.

\begin{cor}
  Under the same hypothesis as for Lemma \ref{lm:var1}, we have
  \begin{equation} \label{eq:var2}
    W' = \frac  14 \int_{\partial N} H' + \langle \II_0,I'\rangle_I da_I~, 
  \end{equation}
  where $\II_0=\II-\frac H2 I$ is the traceless part of $\II$.
\end{cor}

The following lemma is a direct consequence of Lemma \ref{lm:var1}.

\begin{lemma} \label{lm:linear}
Let $r\geq 0$, and let $N_r$ be the set of points of $M$ at distance at most $r$ from $N$. Then $W(N_r) = W(N) - \pi r\chi(\dr M)$. 
\end{lemma}

\begin{proof}
  For $s\in [0,r]$, we denote by $N_s$ be the set of points of $M$ at distance at most $s$ from $N$, and let $w(s)=W(N_s)$. We also denote by $I_s, \II_s, \III_s$ and $B_s$ the induced metric, second and third fundamental forms and the shape operator of $\dr N_s$.
  
According to standard differential geometry formulas, the derivatives of $I_s$ and $\II_s$ are given by:
$$ I_s' =2\II_s~, ~~ \II_s' = \III_s + I_s~. $$ 
Lemma \ref{lm:var1} therefore shows that:
$$ W(N_s)' = \frac{1}4 \int_{\dr N_s} \langle \III_s+I_s-H_s \II_s, I_s\rangle da_s = 
\frac{1}4 \int_{\dr N_s} \tr(B_s^2)+2-H^2_s da_s = $$
$$ = \frac{1}4 \int_{\dr N_s} 2-2\det(B_s) da_s = 
\frac{1}2\int_{\dr N_s} -K da_s = -\pi\chi(\dr N)~. $$
\end{proof}

\subsection{First variation of the $W$-volume from infinity}

We have seen above that given a geodesically convex subset $N\subset M$, we have:
\begin{itemize}
\item the induced metric $I$ and second fundamental form $\II$ on $\partial N$, as well as the shape operator $B$ defined by the condition that $\II=I(B\cdot,\cdot)=I(\cdot, B\cdot)$,
\item the induced metric $I^*$ and second fundamental form $\II^*$ at infinity, as well as the corresponding ``shape operator'' $B^*$, defined by $\II^*=I^*(B^*\cdot,\cdot)=I^*(\cdot, B^*\cdot)$.
\end{itemize}
There is a simple expression of $I^*$ and $\II^*$ from $I$ and $\II$, and conversely, see \eqref{eq:I-I*}.
One can therefore express the first variation of $W$ in terms of the ``data at infinity'' $I^*$ and $\II^*$. A key fact, obtained through a lengthy and not very illuminating computation (see \cite[Lemma 6.1]{volume}) is that Equation \eqref{eq:var1} remains almost identical when expressed in this manner.

\begin{lemma} \label{lm:var1*}
Under a first-order deformation of $N$,
\begin{equation} \label{eq:var1*}
   W' = - \frac{1}{4}\int_{\dr N} \langle {\II^*}' - \frac{H^*}{2} {I^*}',I^*\rangle_{I^*} da_{I^*}~. 
\end{equation}
\end{lemma}

Here $H^*=\tr_{I^*}\II^*$ is the ``shape operator at infinity''.

\begin{cor} \label{cr:var2*}
  Under the same hypothesis, we have
  \begin{equation} \label{eq:var2*}
    W' = - \frac  14 \int_{\partial N} {H^*}' + \langle \II^*_0,{I^*}'\rangle_{I^*} da_{I^*}~, 
  \end{equation}
  where $\II^*_0=\II^*-\frac {H^*}2 I^*$ is the traceless part of $\II^*$ relative to $I^*$.
\end{cor}

\subsection{Definition of the renormalized volume}

Consider a Riemannian metric $h$ on $\partial M$ in the conformal class at  infinity of $\partial_\infty M$. There is by Proposition \ref{pr:33} a unique equidistant  foliation $(S_r)_{r\geq r_0}$ of $M$ near infinity such that the associated metric at infinity is $h$.

For $r\geq r_1$, for a fixed $r_1>0$, the surfaces $S_r$ bound a convex subset of $M$, so that Definition \ref{df:W} applies.

\begin{defi} \label{df:WW}
Let $h$ be a metric on $\dr_\infty M$, in the conformal class at infinity. Let $(S_r)_{r\geq r_0}$ be the equidistant foliation close to infinity associated to $h$. We define $W(M,h) := W(S_r)+ \pi r\chi(\dr M)$, for any choice  of $r\geq r_1$. 
\end{defi}

Lemma \ref{lm:linear} shows that this definition does not depend on the choice of $r\geq r_1$.
As a consequence of the definition, for any $\rho\in \R$, $W(M,e^{2\rho}h)=W(M,h) - \pi\rho \chi(\dr M)$.  

We can now give the definition of the renormalized volume of $M$.

\begin{defi} \label{df:renormvol}
The renormalized volume $V_R$ of $M$ is defined as equal to $W(h)$ when the metric at infinity $h$ is the unique metric of constant curvature $-1$ in the conformal class of $\partial_\infty M$.
\end{defi}

Another possible definition is as the maximum of $W(M,h)$ over all metrics $h$ in the conformal class at infinity of $M$, under the condition that the area of $h$ is equal to $-2\pi \chi(\partial M)$, see \cite{volume}. This is actually an interesting statement: the $W$-volume can be used to simultaneously uniformize the conformal structures at infinity in the asymptotic boundary components of $M$.

\subsection{The variational formula \eqref{eq:dVR}} \label{ssc:variational}




Consider now a first-order deformation of $M$, specified --- through the
Bers Double Uniformization Theorem --- by a first-order deformation of
the conformal structure at infinity, considered as a point in the 
Teichm\"uller space of $\partial M$. 

\begin{prop} \label{pr:variation}
Under a first-order deformation of the hyperbolic structure on $M$,
\begin{equation}
  \label{eq:schlafli-VR}
  dV_R = - \frac{1}{4} \int_{\partial M}\langle \II^*_0, {I^*}'\rangle_{I^*} da_{I^*}~.
\end{equation}
\end{prop}

Here $\langle, \rangle_{I^*}$ is the extension to symmetric 2-tensors of the Riemannian
metric $I^*$ on $T\partial M$. 
Proposition \ref{pr:variation} follows by a simple computation from Equation \eqref{eq:var2*}, see \cite[Lemma 8.5]{volume}, using the fact that at infinity $H^*=-K^*$ (see \cite[Remark 5.4]{volume}, so that if $I^*$ has constant curvature then $\II^*_0$ satisfies the Codazzi equation relative to $I^*$, as $\II^*$ does.

It should be pointed out that Proposition \ref{pr:variation} has a rather simple translation in terms of complex analysis. Since $\II^*_0$ is Codazzi and traceless, it is the real part of a holomorphic quadratic differential, which is minus the Schwarzian derivative $q$ of the uniformization map, see Section \ref{ssc:schwarzian}. Moreover, any first-order deformation $\Id^*$ of the hyperbolic metric at infinity determines a first-order variation of the underlying complex structure, and therefore a Beltrami differential $\mu$.

\begin{cor} \label{cr:dVR}
Equation \eqref{eq:schlafli-VR} can then be written as:
\begin{equation}
  \label{eq:schlafli-simple}
dV_R = - Re\left(\langle q,\mu\rangle\right) = - \int_{\partial M} Re(q\mu)~,
\end{equation}
where $\langle, \rangle$ is the natural pairing between holomorphic quadratic differentials and Beltrami differentials.
\end{cor}

\begin{proof}
The computation needed to go from \eqref{eq:schlafli-VR} to \eqref{eq:schlafli-simple} is local. We choose a complex coordinate $z=x+iy$ adapted to $I^*$, that is, such that $I^*=dx^2+dy^2$ at $z=0$. Let $\mu=(\mu_0+i\mu_1)\frac{\bar{dz}}{dz}$, and $q=(q_0+iq_1)dz^2$. A key point is that $\II^*_0=Re(q)$, see \cite[Appendix A]{volume} (note that the sign here is different because of a different convention in the definition of $q$). Therefore 
$$ \II^*_0=Re(q)=(q_0(dx^2-dy^2)-2q_1dx dy)~, $$
while the first-order variation of $I^*$ is equal to
\begin{eqnarray*}
{I^*}' & = & \frac{d}{dt}_{|t=0}|dz(1+t\mu)|^2 \\
& = & \frac{d}{dt}_{|t=0}|dz+t(\mu_0+i\mu_1)\bar{dz}|^2 \\
& = & 2Re((\mu_0+i\mu_1)\overline{dz}) \\
& = & 2(\mu_0(dx^2-dy^2)+2\mu_1 dx dy)~.  
\end{eqnarray*}
As a consequence, 
\begin{eqnarray*}
\langle \II^*_0,{I^*}'\rangle_{I^*} & = & \langle (q_0(dx^2-dy^2)-2q_1dx dy), 2(\mu_0(dx^2-dy^2)+2\mu_1 dx dy\rangle_{I^*} \\
& = & 4 (\mu_0 q_0 -\mu_1q_1)~,  
\end{eqnarray*}
so that 
$$ \langle \II^*_0,\Id^*\rangle_{I^*} da_{I^*}= 4Re(q\mu)~. $$
The result follows by integrating this equality.
\end{proof}

\subsection{Further properties}

The renormalized volume $V_R$ has other properties which can be very interesting, but will not be considered here because they do not (yet) have any analog on the dual volume side. One key property is that when $c_-$ is fixed, the function $V_R(\cdot, c_-):\cT_{\partial_+M}\to \R$ is a K\"ahler potential for the Weil-Petersson metric on $\cT_{\partial_+M}$. This is proved in \cite[Section 8]{volume} following ideas from \cite{McMullen}.

Another point is that the renormalized volume or closely related functions are generating function that can be used to identify symplectic structures with very different definitions on the space of quasifuchsian manifolds \cite{loustau:minimal}, or to show that certain maps are symplectic (eg \cite{doublemaps}, or \cite{cp} for the grafting map).

Finally, we already noted that the renormalized volume was originally defined in higher dimensions, in the setting of conformally compact Einstein manifolds \cite{Skenderis,graham-witten,graham}. In this setting, some (but not all, so far) of the properties present in 3 dimensions extend nicely, see \cite{renormvol}.

\section{The extremal length and the measured foliation at infinity}

\subsection{The measured foliation at infinity} \label{ssc:foliation}

We now focus on the boundary at infinity of quasifuchsian manifold, and introduce a measured foliation which can be thought of as an analog at infinity of the measured bending lamination on the boundary of the convex core.

\begin{defi}
The {\em measured foliation at infinity} is the horizontal measured foliation of $q$, the Schwarzian derivative of the uniformization map at infinity. We denote it by $f$.
\end{defi}

\begin{proof}[Proof of Theorem \ref{tm:ext}]
  According to Lemma \ref{lm:Q}, the extremal length $\ext_{c_\pm}f_{\pm}$ is the integral over $\partial_{\infty,\pm}$ of $|q|$. By the Nehari estimate (Theorem \ref{tm:nehari}), $|q|\leq 3da_{h_\pm}/2$, where $da_{h_\pm}$ is the area form of the hyperbolic metric $h_\pm$ compatible with $c_\pm$. The result follows. 
\end{proof}

We now consider one connected component of the ideal boundary of $M$, say $\partial_{\infty,+}M$, equiped with its canonical conformal structure. Recall from Section \ref{ssc:energy} that $T_f$ is the real tree dual to the universal cover of the measured foliation $f$, and that $\Phi_f$ is the Hopf differential of the unique equivariant harmonic map from the universal cover of $\partial_{\infty,+}M$, equiped with this conformal structure, to $T_f$. The same construction works for $\partial_{\infty,-}M$.

It follows from Theorem \ref{tm:wolf} that $\Phi_{f}=-q$. 

The following lemma relates the renormalized volume to the measured foliation at infinity.

\begin{lemma} \label{lm:critical}
Let $c\in \cT_{\partial M}$, and let $F\in \cMF_{\partial M}$.  Then $F$ is the measured foliation at infinity of the quasifuchsian hyperbolic metric determined by $c$ if and only if the function
$\Psi_F$ defined as $$ \Psi_F = V_R - \frac 14 E_F: \cT_{\partial M}\to \R $$
is critical at $c$.
\end{lemma}

\begin{proof}
  Suppose first that $F$ is the horizontal measured foliation of $q$, the holomorphic quadratic differential at infinity of the quasifuchsian manifold $M(c)$. 

It follows from \eqref{eq:gardiner} and \eqref{eq:schlafli-simple} that, in a first-order variation $\dot c$, 
$$ d\Psi_F(\dot c) = dV_R(\dot c) - \frac 14 dE_F(\dot c) = \re(\langle q+\Phi_F, \dot c\rangle)~. $$
But it follows from Theorem \ref{tm:wolf} that $q=-\Phi_F$, and it follows that $d\Psi_F=0$.

Conversely, if $d\Psi_F=0$, the same argument as above shows that $q=-\Phi_F$, so that $F$ is the horizontal measured foliation of $q$.
\end{proof}

\subsection{Proof of Theorem \ref{tm:schlafli}}

According to Equation \eqref{eq:schlafli-simple}, in a first-order deformation of $M$,
$$ \dot V_R = - \re(\langle q, \dot c\rangle )~, $$
and using Theorem \ref{tm:wolf} we obtain that
$$ \dot V_R = \re(\langle \Phi_f, \dot c\rangle )~. $$
Using \eqref{eq:gardiner}, this can be written as
$$ \dot V_R = - \frac 14 dE_f(\dot c)~. $$
Using Theorem \ref{tm:wolf} again, we finally find that
$$ \dot V_R = - \frac 12 (d\ext(f))(\dot c)~. $$


\section{Comparisons}

\subsection{Outline}

Some applications of the renormalized volume follow from the following related facts, each having its own independent proof.
\begin{enumerate}
\item The dual volume of the convex core is within a bounded additive constant (depending only on the genus) from the volume of the convex core.
\item The dual volume is within a bounded additive constant (depending only on the genus) from the renormalized volume.
\item The renormalized volume is bounded from above by the Weil-Petersson distance between the conformal metrics $c_-, c_+$ on the connected components of its boundary at infinity (times an explicit constant).
  \item The dual volume is bounded from above by the Weil-Petersson distance between the induced metrics $m_-, m_+$ on the two boundary components of the convex core (times an explicit function). 
\end{enumerate}
We outline the main arguments --- and provide references --- for those statements below.

\subsection{Comparing the renormalized volume and the dual volume}

The renormalized volume can be compared to the dual volume using the following statement, see \cite[Prop. 3.12]{compare_v4}.

\begin{lemma}
  Let $h,h'$ be two metrics on $\partial_\infty M$, in the conformal class at infinity. Suppose that $h'\geq h$ at each point. Then $W(M,h')\geq W(M,h)$, with equality if and only if $h=h'$.
\end{lemma}

The proof of this lemma rests on the fact that if $h'\geq h$, then whenever $r>0$ is such that the equidistant surfaces $S_r$ and $S'_r$ associated to $h$ and $h'$, respectively, are well-defined (see Section \ref{ssc:equidistant}), then $S'_r$ is $S_r$ is in the interior of $S'_r$. And moreover if $S_r$ is in the interior of $S'_r$, and both $S_r$ and $S'_r$ bound convex subsets, then the $W$-volume of the domain bounded by $S_r$ is smaller than the $W$-volume of the domain bounded by $S'_r$ --- $W$ is increasing under inclusion of convex subsets.

By definition of the $W$-volume, we see that $W(C(M))=W(M,h)$ where $h$ is the metric at infinity defined by the folation of $M\setminus C(M)$ by surfaces equidistant to $C(M)$. A direct computation (see \cite{compare_v4}) shows that this metric is equal to $h=h_{Th}/2$, where $h_{Th}$ is {\em Thurston's projective metric}. This metric $h_{Th}$ has a simple description when the bending lamination $l$ is supported on simple closed curves: is it then obtained by cutting the induced metric on the boundary of the convex core along the bending curves and inserting for each a flat cylinder of width equal to the exterior bending angle. A key feature of this metric is that it is in the conformal class at infinity $c$ of $M$, see \cite{kulkarni-pinkall,kamishima-tan}.

Let $h_{-1}$ be the hyperbolic metric in the conformal class $c$ at infinity. Then 
\begin{equation}
  \label{eq:anderson}
  h_{-1}\leq h_{Th}\leq 2h_{-1}~. 
\end{equation}
The first inequality follows from the definition of the Thurston metric, or from the fact that $h_{Th}$ has curvature at least $-1$ at all points. The second inequality is a direct consequence of a a result of G. Anderson \cite[Theorem 4.2]{anderson1998}, see \cite[Theorem 2.1]{bridgeman-brock-bromberg}. 

It is also useful to remark that if $N\subset M$ is geodesically convex, and if the metric at infinity associated to $N$ is $h$, then for $r>0$ the metric at infinity associated to $N_r$ is $e^{2r}h$. So it follows from Lemma \ref{lm:linear} that
$$ W(M,e^{2r}h) = W(M,h)-\pi r \chi(\partial M)~. $$
It therefore follows from Equation \eqref{eq:anderson} that:
$$ W(M,h_{-1}) \leq W(M, h_{Th}) \leq W(M, h_{-1}) - \frac{\pi \log(2)}2 \chi(\partial M)~, $$
so that
$$ V_R(M) \leq W(C(M)) + \frac{\pi \log(2)}2 \chi(\partial M) \leq V_R(M) - \frac{\pi \log(2)}2 \chi(\partial M)~. $$
Recall that
$$ W(C(M))=V(C(M))-\frac 14 L_m(l) = V^*_C(M) + \frac 14 L_m(l)~, $$
so we obtain that
$$ V_R(M) \leq V^*_C(M) + \frac 14 L_m(l)+ \frac{\pi \log(2)}2 \chi(\partial M) \leq V_R(M) - \frac{\pi \log(2)}2 \chi(\partial M)~. $$
Finally, it is known that $L_m(l)\leq 6\pi|\chi(\partial M)|$ (see \cite[Theorem 1.1 (2)]{bridgeman-brock-bromberg}, and we therefore obtain the following statement.

\begin{theorem} \label{tm:VRV*}
For all quasifuchsian metric on $M$,
$$ V_R(M) - \frac {3\pi |\chi(\partial M)|}2 + \frac{\pi \log(2)}2 |\chi(\partial M|) \leq V^*_C(M) \leq V_R(M) + \pi \log(2) |\chi(\partial M|)~. $$
\end{theorem}

The additive constants depend on the choice of normalization in the definition of the renormalized volume --- chosing a metric at infinity of constant curvature $-2$, rather than $-1$, leads to somewhat simpler additive constants.

\subsection{An upper bound on the renormalized volume}

The renormalized volume of a quasifuchsian manifold can be bounded from above in terms of the Weil-Petersson distance between the conformal metrics on $\partial_{\infty,-}M$ and on $\partial_{\infty,+}M$. This upper bound is based on the following classical result. We denote by $D$ the unit disk in $\C$, equiped with the hyperbolic metric $h$.

\begin{theorem}[Kraus, Nehari \cite{nehari-bams}] \label{tm:nehari}
Let $f:D\to \CP^1$ be an injective holomorphic map. Then at each point $\| \cS(f)\|_h\leq 3/2$.
\end{theorem}

The following theorem from \cite{compare_v4} is a direct consequence.

\begin{theorem} \label{tm:upper}
For any quasifuchsian metric $g_0$ on $S\times \R$,
\begin{equation}
V_R(g_0) \leq 3\sqrt{\pi (g-1)} \dwp(c_-,c_+)~, 
\end{equation}
where $c_-$ and $c_+$ are the conformal structures at infinity of $g_0$ and 
$d_{WP}$ is the Weil-Petersson distance.
\end{theorem}

\begin{proof} 
Let $c\in \cT_S$ be a complex structure on $S$, let 
$q$ and $\mu$ be a holomorphic quadratic differential and a Beltrami differential 
on $(S,c)$, and let $h'$ be the first-order variation corresponding to $\mu$
of the hyperbolic metric $h$ in the conformal class defined by $c$. Then a 
direct computation shows that
$$ \int_S \langle Re(q),h'\rangle_h da_h = 4Re\left(\int_S q\mu\right)~. $$
Applying this relation with $q$ equal to Schwarzian derivative term as above,
and using that $\II^*_0=-Re(q)$, we obtain that for a variation $h'$ of the
hyperbolic metric $h$ in the conformal class on the upper component of the boundary
at infinity, 
$$ dV_R(h') = -\frac 14\int_S \langle \II^*_0,h'\rangle_h da_h = 
\frac 14\int_S\langle Re(q),h'\rangle_h da_h =
Re\left( \int_S q\mu \right)~. $$

Let $z$ be a local complex coordinate, with $h=\rho^2|dz|^2$, then 
we can write
$$ q = q' dz^2~,~~ \mu = \mu' \frac{d\bar z}{dz}~, $$
so that 
$$ dV_R(h') = Re\left(\int_S \left(\frac{q'}{\rho^2}\right) \mu'
\rho^2 |dz|^2\right)~. $$
Using the Nehari estimate (Theorem \ref{tm:nehari}) shows that $|q'/\rho^2|\leq 3/2$, and therefore
$$ |dV_R(h')|\leq \frac 32 \int_S \left| \mu' \right| \rho^2 |dz|^2~. $$
It then follows from the Cauchy-Schwarz inequality that
$$ |dV_R(h')|\leq \frac 32 \| \mu\|_{WP} \sqrt{4\pi (g-1)} = 3\sqrt{\pi (g-1)} \| \mu\|_{WP}~. $$

We can integrate this inequality on a path from $c_-$ to $c_+$ as in the first proof above to obtain the result.
\end{proof}

\subsection{An upper bound on the dual volume}

We have seen in Theorem \ref{tm:VRV*} that the dual volume $V_C^*(M)$ is within a bounded additive constant from the renormalized volume $V_R(M)$. In addition, Theorem \ref{tm:upper} shows that $V_R(M)$ is bounded by $3\sqrt{\pi (g-1)} \dwp(c_-,c_+)$. This immediately yields the following corollary.

\begin{cor} \label{cr:upper*}
  For all quasifuchsian metric on $M$,
  \begin{equation}
    \label{eq:V*}
      V_C^*(M) \leq 3\sqrt{\pi (g-1)} \dwp(c_-,c_+) + \pi \log(2) |\chi(\partial M|)~. 
  \end{equation}
\end{cor}

It should be noted, however, that this argument is quite indirect and uses the whole technology of the renormalized volume.

\subsection{Estimates from the dual volume}

Recently, Filippo Mazzoli \cite{mazzoli2019} has developed a completely different and much more elementary argument to obtain directly an inequality of the type of \eqref{eq:V*}, and as a consequence an explicit upper bound on $V_C(M)$ in terms of the Weil-Petersson distance between $c_-$ and $c_+$, using the {\em dual Bonahon-Schl\"afli formula}.

\begin{theorem}[Dual Bonahon-Schl\"afli formula]
  Under a first-order variation of a quasifuchsian structure on $M$,
  \begin{equation}
    \label{eq:BS*2}
    {V_C^*}' = -\frac 12 dL(l)(m')~. 
  \end{equation}
\end{theorem}

A proof of this formula can be found in \cite[Lemma 2.2]{cp}, based on an analoguous formula proved by Bonahon \cite{bonahon,bonahon-variations}: under the same hypothesis,
\begin{equation}
  \label{eq:BS2}
  V_C' = \frac 12 L_m(l')~. 
\end{equation}
Note however that \eqref{eq:BS*} has a much simpler interpretation than \eqref{eq:BS}, since \eqref{eq:BS*} involves only the differential of the (analytic) function $L(l)$ applied to the tangent vector $m'$, while \eqref{eq:BS} uses the notion of first-order variation of a measured lamination, notion which is quite subtle and necessitates the full technical toolbox developed by Bonahon \cite{bonahon}.

A direct and relatively elementary (but non-trivial) proof of \eqref{eq:BS*} is given by Mazzoli \cite{mazzoli2018}, using differential-geometric arguments and an approximation of the boundary of the convex by smooth surfaces.

Mazzoli then shows \cite{mazzoli2019} that \eqref{eq:BS*} can be used, together with Theorem \ref{tm:6pi}, to obtain directly an upper bound on the dual volume.

\begin{theorem}[Mazzoli \cite{mazzoli2019}] \label{tm:mazzoli}
  There exists a constant $K_2>0$ such that for all quasifuchsian manifold $M$ and all first-order deformation, 
  $$ |{V_C^*}'| \leq K_2 \sqrt{g-1} \| c'\|_{WP}~. $$
\end{theorem}

It follows directly from this inequality --- as in the proof of Theorem \ref{tm:upper} above --- that for any quasifuchsian manifold,
\begin{equation}
  \label{eq:VC*d}
  V_C^*(M) \leq K_2 \sqrt{g-1} d_{WP}(c_-, c_+)~.
\end{equation}
The constant found in \cite{mazzoli2019} is $K_2=10.3887$, which is slightly larger than the constant obtained for the renormalized volume in Theorem \ref{tm:upper}.

\section{Applications}
\label{sc:applications}

We briefly outline in this section a few applications of the bound on the renormalized volume, or the dual volume, of quasifuchsian manifolds. This section does not contain complete proofs --- we refer to specific papers for the details --- but only a very quick outline of the main ideas. 

\subsection{Bounding the volume of the convex core using the renormalized volume or the dual volume}

We have seen in Theorem \ref{tm:VRV*} that the renormalized volume $V_R(M)$ is within bounded additive constant (depending only on the genus of the underlying surface) from the dual volume $V_C^*(M)$, while Theorem \ref{tm:upper} provides an upper bound on the renormalized volume in terms of the Weil-Petersson distance between the conformal metrics at infinity. It follows directly that the volume of the convex core is also bounded in terms of the Weil-Petersson distance between the conformal metrics at infinity: for every genus $g>1$, there exists a constant $C_g>0$ such that for all quasifuchsian manifold $M$,
$$ V_C(M)\leq 3\sqrt{\pi (g-1)} \dwp(c_-,c_+) + C_g~. $$

It also follows from Theorem \ref{tm:mazzoli}, thanks to the upper bound on the length of the bending lamination in Theorem \ref{tm:6pi}, that the same inequality holds for $V_C(M)$, with an additional term $3\pi |\chi(M)|$: for all quasifuchsian manifold $M$,
$$ V_C(M) \leq K_2 \sqrt{g-1} d_{WP}(c_-, c_+) + 3\pi|\chi(\partial M)|~. $$
However at this point the constant $K_2$ arising from Mazzoli's work \cite{mazzoli2019} is somewhat weaker than the $3\sqrt{\pi}$ coming out of the renormalized volume argument.

\subsection{The volume of hyperbolic manifolds fibering over the circle}

A neat applications of the upper bound found in the previous section on the volume of the convex core is given by Kojima and McShane \cite{kojima-mcshane} and Brock and Bromberg \cite{brock-bromberg:inflexibility2}.

Let $\phi:S\to S$ be a diffeomorphism. The {\em mapping torus} of $\phi$ is the 3-dimensional manifold $M_\phi$ obtained by identifying in $S\times [0,1]$ the points $(x,1)$ and $(\phi(x),0)$ for all $x\in S$. Clearly, $\phi$ depends only on the isotopy class of $\phi$. Thurston \cite{thurston:hyperbolicII} proved that if $\phi$ is {\em pseudo-Anosov} (see e.g. \cite{FLP}) then $M_\phi$ admits a hyperbolic structure, which is unique by the Mostow Rigidity Theorem \cite{mostow:ihes}.

A pseudo-Anosov diffeomorphism $\phi$ acts by pull-back on the Teichm\"uller space $\cT_S$, and this action is isometric for the Weil-Petersson metric. One can define its Weil-Petersson {\em translation length}. 
$$ l(\phi) = \min_{c\in \cT_S} d_{WP}(c,\phi_*(c))~. $$
Moreover, this minimum is attained along a line, the {\em axis} of $\phi$, on which it acts by translation, see \cite{daskadopoulos-wentworth}.

\begin{theorem}[Kojima-McShane, Brock-Bromberg] \label{tm:lphi}
  For any pseudo-Anosov diffeomorphism $\phi$ of $S$, $Vol(M_\phi)\leq 3\sqrt{\pi(g-1)} l(\phi)$.
\end{theorem}

The proof of this theorem parallels the construction by Thurston of the hyperbolic structure on $M_\phi$. Given $c_-,c_+\in \cT_S\times \cT_S$, let $M(c_-,c_+)$ be the quasifuchsian hyperbolic structure on $S\times \R$ with conformal structure at infinity $c_-$ and $c_+$, respectively. If $c$ is any fixed element of $\cT_S$, Thurston proved that $M(\phi^{-n}_*c,\phi^n_*c)\to \bar M$, the infinite cyclic cover of $M$. Brock and Bromberg \cite{brock-bromberg:inflexibility2}, building on work of McMullen \cite{mcmullen:renormalization}, show that this convergence translates as a precise estimate on the volume of the convex core.

\begin{theorem}[Brock--Bromberg \cite{brock-bromberg:inflexibility2}]
  In this setting, $|V_C(\phi^{-n}_*c, \phi_*^nc)-2nVol(M_\phi)|$ is bounded.
\end{theorem}

It follows that
$$ V_C(\phi^{-n}_*c, \phi_*^nc) \leq 3\sqrt{\pi(g-1)}d_{WP}(\phi^{-n}_*c, \phi_*^nc)+C_g~. $$
Taking for $c$ an element of the axis of $\phi$, we obtain that
$$ V_C(\phi^{-n}_*c, \phi_*^nc) \leq 6n\sqrt{\pi(g-1)}l_\phi+C_g~. $$
As a consequence,
$$ Vol(M_\phi) = \lim_{n\to \infty} \frac{V_C(\phi^{-n}_*c, \phi_*^nc)}{2n} \leq 3\sqrt{\pi(g-1)} l_\phi~, $$
which is Theorem \ref{tm:lphi}

\subsection{Systoles of the Weil-Petersson metric on moduli space}

As a consequence of Theorem \ref{tm:lphi}, Brock and Bromberg obtain a lower bound on the systole $\cM_S$ of the moduli space of $S$, equipped with the Weil-Petersson metric.

\begin{cor}[Brock--Bromberg]
  The shortest closed geodesic of the Weil-Petersson metric on $\cM_S$ has length at least $Vol(\cW)/3\sqrt{\pi(g-1)}$.
\end{cor}

Here $\cW$ is the Weeks manifold, the closed hyperbolic manifold of smallest volume. This statement follows from Theorem \ref{tm:lphi} and from the fact that any closed geodesic of moduli space corresponds to a pseudo-Anosov element of the mapping-class group.

\subsection{Entropy and hyperbolic volume of mapping tori}

Let $\phi:S\to S$ be a diffeomorphism. We denote here by $\ent(\phi)$ the {\em entropy} of $\phi$, that is, the infimum of the topological entropy of diffeomorphisms isotopic to $\phi$. If $\phi$ is a pseudo-Anosov diffeomorphism, Thurston showed that its entropy is equal to the log of the minimal dilation of diffeomorphisms isotopic to $\phi$ \cite[Exposé 10]{FLP}, and Bers \cite{bers:thurston} proved that this is equal to its translation length for the Teichm\"uller distance on $\cT_S$.

Kojima and McShane prove the following relation between the entropy of $\phi$ and the hyperbolic volume of the mapping torus of $\phi$.

\begin{theorem}\label{tm:kojima-mcshane}
  If $\phi$ is a pseudo-Anosov diffeomorphism of $S$, then
  $$ \ent(\phi)\geq \frac 1{3\pi|\chi(S)|} \vol(M_\phi)~. $$
\end{theorem}

The proof of this result is quite similar to the proof of Theorem \ref{tm:lphi} above, but the bound on the renormalized volume of a quasifuchsian manifold by the Weil-Petersson distance between its conformal metrics at infinity is replaced by a bound by the Teichm\"uller distance between those conformal metrics at infinity. Specifically, Kojima and McShane prove the following statement, see \cite[Prop. 11]{kojima-mcshane}.

\begin{theorem}[Kojima--McShane]
  Let $c_-, c_+\in \cT_S$. Then
  $$ V_R(M(c_-, c_+)) \leq 3\pi |\chi(S)| d_T(c_-, c_+)~, $$
  where $d_T$ denotes the Teichm\"uller distance on $\cT_S$.
\end{theorem}

The proof is closely related to the proof of Theorem \ref{tm:upper}. We consider a Teichm\"uller geodesic $(c_t)_{t\in [0,1]}$ with $c_0=c_-$ and $c_1=c_+$. Using Corollary \ref{cr:dVR} and the Nehari estimate, Theorem \ref{tm:nehari}, we obtain that:
\begin{eqnarray*}
  V_R(M(c_-, c_+)) & \leq & \int_0^1 |V_R(c_0, c_t)'| dt \\
                   & \leq & \int_0^1 \left(\int_{\partial_{\infty,+}M} |\re(q\mu_t)| \right) dt\\
                   & \leq & \int_0^1 \left(\int_{\partial_{\infty,+}M} \frac 32|\mu_t| da_h\right) dt \\
  & \leq & 3\pi|\chi(S)| d_T(c_0, c_1)~. 
\end{eqnarray*}

\section{Questions and perspectives}
\label{sc:questions}

We list here a number of questions concerning the global behavior of the renormalized volume, and in a related way of the measured foliations at infinity, for quasifuchsian manifolds and generalizations.

\subsection{Convexity of the renormalized volume}

The renormalized volume is known to be convex in the neighborhood of the Fuchsian locus, see \cite{moroianu-convexity,ciobotaru-moroianu,vargas-pallete-local}. However, not much is known on its global behavior. 

\begin{question} \label{q:convexity}
Is the renormalized volume convex with respect to the Weil-Petersson metric on $\cT_{\partial M}$?
\end{question}

It would of course be interesting to know whether the renormalized volume is convex in any other sense, for instance along Teichm\"uller geodesics or earthquake deformations.

A related question is whether there is an explicit {\em lower} bound on the renormalized volume in terms of the Weil-Petersson distance between the conformal metrics at infinity of a quasifuchsian manifold. The existence of such a constant in fact follows from the results of Brock \cite{brock:2003} on the volume of the convex core, together with Theorem \ref{tm:VRV*}, but no estimate of this constant is known.

\begin{question} \label{q:lower}
  Let $g\geq 2$. What is the largest constant $c_g>0$ for which there exists a constant $d>0$ such that, for all quasifuchsian structure $g$ on $S\times \R$ (where $S$ is a closed surface of genus $g$) with conformal metrics at infinity $c_-, c_+\in \cT_S$,
  $$ V_M(M,g)\geq c_g d_{WP}(c_-, c_+)-d~? $$
\end{question}

Note that it has been proved recently that the renormalized volume is minimal at the Fuchsian locus (for quasifuchsian manifolds) and for metrics containing a convex core with totally geodesic boundary (for acylindrical manifolds), see \cite{vargas-pallete-continuity,bridgeman-brock-bromberg}. 

Note that Theorem \ref{tm:VRV*} shows that Question \ref{q:lower} is equivalent to the corresponding question for the volume or the dual volume of the convex core.

\subsection{The measured foliation at infinity}

The analogy between the measured folation at infinity and the measured bending lamination on the boundary of the convex core suggests to extend to the foliation at infinity a number of statements known or conjectures on the bending measured lamination on the boundary of the convex core. The first question in this direction can be the following.

\begin{question} \label{q:fill}
Suppose that $M$ is not Fuchsian (that is, it does not contain a closed totally geodesic surface). Do $f_-$ and $f_+$ fill?   
\end{question}

This would be the analog of the well-known (and relatively easy) corresponding statement for $l_-$ and $l_+$, the measured bending lamination on the boundary of the convex core.

\begin{question} 
Let $(f_-,f_+)\in \cML_S\times \cML_S$, $(f_-, f_+)\neq 0$. Is there at most one quasifuchsian manifold with measured foliation at infinity $(f_-, f_+)$? 
\end{question}

This is the analog at infinity of the uniqueness part of a conjecture of Thurston on the existence and uniqueness of a quasifuchsian manifold having given measured bending lamination $(l_-, l_+)$ on the boundary of the convex core. In this case $(l_-, l_+)$ are requested to fill and to have no closed leaf of weight larger than $\pi$. The existence part of this conjecture for the bending measured lamination was proved in \cite{bonahon-otal}, as well as the uniqueness for rational measured laminations, but the uniqueness remains conjectural for more general measured laminations.

A related question would be whether {\em infinitesimal} rigidity holds, that is, whether any non-zero first-order deformation of $M$ induces a non-zero deformation of either the $f_-$ or $f_+$ --- this might be related to Question \ref{q:convexity}. The analog question for $l_-$ and $l_+$ is also open.

One can also ask for what pair $(f_-, f_+)$ of measured foliations there {\em exists} a quasifuchsian manifold having them as measured foliation at infinity:

\begin{question}
Given $(f_-,f_+)\in \cML_S\times \cML_S$, what conditions should it satisfy so that there exists a quasifuchsian manifold $M$ with measured foliation at infinity $(f_-, f_+)$? 
\end{question}

If the answer to Question \ref{q:fill} is positive, then one should ask that (if $(f_-,f_+)\neq 0$) $f_-$ and $f_+$ should fill. However other conditions might be necessary.



\subsection{Comparing the foliation at infinity to the measured bending lamination}

We have seen above that the renormalized volume of a quasifuchsian is within a bounded additive distance (depending on the genus) from the dual volume of the convex core, and also that the induced metric on the boundary of the convex core is within a bounded quasi-conformal constant of the conformal metric at infinity.

This suggests the following question.

\begin{question}
Is the measured foliation at infinity of a quasifuchsian manifold within bounded distance --- in a suitable sense --- from the measured bending lamination on the boundary of the convex core?  
\end{question}

A recent result of Dumas \cite{dumas-duke} should be relevant here and actually provides a kind of answer.

\subsection{Extension to convex co-compact or geometrically finite hyperbolic manifolds}
\label{ssc:extension}

The definition of the renormalized volume can be extended to convex co-compact hyperbolic manifolds, and the main estimates also apply for convex co-compact manifolds with incompressible boundary, see \cite{bridgeman-canary:renormalized}. We can expect Theorem \ref{tm:schlafli} to apply to convex co-compact hyperbolic manifolds, and Theorem \ref{tm:ext} to extend to convex co-compact hyperbolic manifolds with incompressible boundary, while the estimate for manifolds with compressible boundary might involve the injectivity radius of the boundary.

\begin{question}
Can Theorem \ref{tm:ext} and Theorem \ref{tm:schlafli} be extended to geometrically finite hyperbolic 3-manifolds?
\end{question}

Again, the definition and some key properties of the renormalized volume extend to geometrically finite hyperbolic 3-manifolds, see \cite{guillarmou-moroianu-rochon}. It could be expected that Theorems \ref{tm:schlafli} and \ref{tm:ext} extends to this setting.

\subsection{Higher dimensions}

\begin{question}
Are there any extensions of the measured foliation at infinity in higher dimension, for quasifuchsian (or convex co-compact) hyperbolic $d$-dimensional manifolds? 
\end{question}

For those manifolds, there is a well-defined notion of convex core, and the boundary of the convex core also has a ``pleating''. However the pleating lamination might have a more complex structure than for $d=3$, with codimension $1$ ``pleating hypersurfaces'' of the boundary meeting along singular strata of higher codimension. Other aspects of the renormalized volume of quasifuchsian manifold have a partial extension in higher dimensions, see e.g. \cite{renormvol}.

\bibliographystyle{plain}
\bibliography{/home/schlenker/Dropbox/papiers/outils/biblio}

\end{document}

%% file: macros.tex
\usepackage{mathrsfs}
\usepackage{mathtools}
\usepackage{stmaryrd}

\usepackage{amssymb,amscd}
\usepackage{amsmath,amsfonts,latexsym}
\usepackage{color}
\usepackage[english]{babel}
\usepackage{tikz}
\usetikzlibrary{matrix,arrows}

\newcounter{notes}%

\newtheorem{cor}{Corollary}[section]
\newtheorem{theorem}[cor]{Theorem}
\newtheorem{prop}[cor]{Proposition}
\newtheorem{lemma}[cor]{Lemma}

\newtheorem{conj}[cor]{Conjecture}

\newtheorem{stat}[cor]{Statement}

\usepackage{amsthm}

\makeatletter
\newtheorem*{rep@theorem}{\rep@title}
\newcommand{\newreptheorem}[2]{%
\newenvironment{rep#1}[1]{%
 \def\rep@title{#2 \ref{##1}}%
 \begin{rep@theorem}}%
 {\end{rep@theorem}}}
\makeatother

\newreptheorem{theorem}{Theorem}

\theoremstyle{definition}
\newtheorem{defi}[cor]{Definition}

\newtheorem{question}[cor]{Question}

\theoremstyle{plain}

\newcommand{\cF}{{\mathcal F}}

\newcommand{\cM}{{\mathcal M}}

\newcommand{\cP}{{\mathcal P}}

\newcommand{\cS}{{\mathcal S}}
\newcommand{\cT}{{\mathcal T}}
\newcommand{\cQ}{{\mathcal Q}}
\newcommand{\cW}{{\mathcal W}}

\newcommand{\cCP}{{\mathcal C\mathcal P}}
\newcommand{\cMF}{{\mathcal M\mathcal F}}
\newcommand{\cML}{{\mathcal M\mathcal L}}

\newcommand{\cQF}{{\mathcal Q\mathcal F}}

\newcommand{\C}{{\mathbb C}}
\newcommand{\HH}{{\mathbb H}}

\newcommand{\R}{{\mathbb R}}

\newcommand{\dr}{{\partial}}

\newcommand{\hd}{\dot{h}}

\newcommand{\vol}{\mbox{Vol}}

\newcommand{\PSL}{\mathrm{PSL}}

\newcommand{\ent}{\mathrm{ent}}

\newcommand{\Id}{\mathrm{I}}

\newcommand{\tr}{\mbox{\rm tr}}

\newcommand{\ext}{\mbox{ext}}

\newcommand{\II}{I\hspace{-0.1cm}I}
\newcommand{\III}{I\hspace{-0.1cm}I\hspace{-0.1cm}I}

\newcommand{\CP}{\mathbb{CP}}

\newcommand{\RP}{\mathbb{RP}}

